\documentclass[12pt]{article}
\usepackage{graphicx}
\usepackage{amsmath,amsthm,amssymb,amsfonts,euscript,enumerate}
\usepackage{amsthm}
\usepackage{color,wrapfig}
\usepackage{pxfonts}
\usepackage{verbatim}
\newtheorem{thm}{Theorem}[section]

\newtheorem{lem}[thm]{Lemma}
\newtheorem{prop}[thm]{Proposition}

\newcommand{\1}{\partial}
\newcommand{\2}{\overline}
\newcommand{\3}{\varepsilon}

\newcommand{\R}{{\mathbb R}}

\newcommand{\Z}{{\mathbb Z}}

\newcommand{\bs}{\backslash}

\newcommand{\La}{\Delta}

\begin{document}
\title{Uniqueness and time oscillating behaviour of finite\\ 
points blow-up solutions of the fast diffusion equation}
\author{Kin Ming Hui\\
Institute of Mathematics, Academia Sinica\\
Taipei, Taiwan, R. O. C.}
\date{May 26, 2018}
\smallbreak \maketitle
\begin{abstract}
Let $n\ge 3$ and $0<m<\frac{n-2}{n}$. We will extend the results of J.L. Vazquez and M. Winkler \cite{VW2} and prove the uniqueness of finite points blow-up solutions of the fast diffusion equation $u_t=\Delta u^m$ in both bounded domains and $\R^n\times (0,\infty)$. We will also construct initial data such that the corresponding solution of the fast diffusion equation in bounded domain oscillate between infinity and some positive constant as $t\to\infty$.
\end{abstract}

\vskip 0.2truein



\vskip 0.2truein
\setcounter{equation}{0}
\setcounter{section}{0}

\section{Introduction}
\setcounter{equation}{0}
\setcounter{thm}{0}

The equation
\begin{equation}\label{fde}
u_t=\Delta u^m
\end{equation}
arises in many physical and geometrical models \cite{A}, \cite{DK}, \cite{V1}, \cite{V2}. When $m>1$, \eqref{fde}  is called porous medium equation which appears in the modeling of the flow of  gases through porous media and oil passing through sand etc. \eqref{fde} also arises as the large time asymptotic limit solution of the compressible Euler equation with damping \cite{HPW}, \cite{LZ}. When $m=1$, \eqref{fde} is the heat equation.  When  $0<m<1$,  \eqref{fde} is called the fast diffusion equation. When $m=\frac{n-2}{n+2}$ and $n\ge 3$,  \eqref{fde} arises in the study of Yamabe flow on $\R^n$ \cite{DS1}, \cite{DS2}, \cite{PS}. Note that the metric $g_{ij}=u^{\frac{4}{n+2}}dx^2$, $u>0$, $n\ge 3$, is a solution of the Yamabe flow  \cite{DS2}, \cite{PS},
\begin{equation*}
\frac{\1 g_{ij}}{\1 t}=-Rg_{ij} \quad \mbox{ in }\R^n\times (0,T)
\end{equation*}
if and only if  $u$ is a solution of 
\begin{equation*}
u_t=\frac{n-1}{m}\Delta u^m
\end{equation*}
in $\R^n\times (0,T)$
with $m=\frac{n-2}{n+2}$ where $R(\cdot,t)$ is the scalar curvature of the metric $g_{ij}(\cdot,t)$.    Recently F.~Golse and F.~Salvarani \cite{GS}, B.~Choi and K.~Lee \cite{CL}, have shown that \eqref{fde} also appears as the nonlinear diffusion limit for the generalized Carleman models.

Although there are a lot of study of \eqref{fde} for the case $m>\frac{(n-2)_+}{n}$, there are not many results of \eqref{fde} for the case $0<m<\frac{n-2}{n}$, $n\ge 3$. When $0<m\le\frac{n-2}{n}$ and $n\ge 3$,
existence of positive smooth solutions of 
\begin{equation*}
\left\{\begin{aligned}
u_t=&\Delta u^m, u\ge 0,\quad\mbox{ in }\R^n\times (0,T)\\
u(x,0)=&u_0\qquad\qquad\quad\mbox{ in }\R^n
\end{aligned}\right.
\end{equation*}
for any $0\le u_0\in L_{loc}^p(\R^n)$, $p>(1-m)n/2$, satisfying the condition,
\begin{equation*}
\liminf_{R\to\infty}\frac{1}{R^{n-\frac{2}{1-m}}}\int_{|x|\le R}u_0\,dx
\ge C_1T^{\frac{1}{1-m}}
\end{equation*}
for some constant $C_1>0$ is proved by S.Y.~Hsu in \cite{Hs}.

Let $\Omega\subset\R^n$ be a smooth bounded domain. When $0<m\le\frac{n-2}{n}$, $n\ge 3$ and $0\in\Omega$, existence of singular solutions and asymptotic large time behaviour of \eqref{fde} in $(\Omega\setminus\{0\})\times (0,\infty)$ which blows up at $\{0\}\times (0,\infty)$ when the initial value $u_0$ satisfies 
\begin{equation*}
c_1|x|^{-\gamma_1}\le u_0(x)\le c_2 |x|^{-\gamma_2}\quad\forall x\in\Omega\setminus\{0\}
\end{equation*}
for  some constants $c_1>0$, $c_2>0$ and $\gamma_2\ge\gamma_1>\frac{2}{1-m}$ were proved by J.L. Vazquez and M. Winkler in \cite{VW1}. Uniqueness of singular solutions of \eqref{fde} in $(\Omega\setminus\{0\})\times (0,\infty)$ that blows up at $\{0\}\times (0,\infty)$ and existence of singular initial data such that the corresponding singular solution of \eqref{fde} in $(\Omega\setminus\{0\})\times (0,\infty)$ oscillates between infinity and some positive constant as $t\to\infty$ are proved by J.L. Vazquez and M. Winkler in \cite{VW2}.  

When $0<m\le\frac{n-2}{n}$ and $n\ge 3$, existence  of singular solutions of \eqref{fde} in $(\R^n\setminus\{0\})\times (0,\infty)$ which blows up at $\{0\}\times (0,\infty)$  when the initial value $u_0$ satisfies 
\begin{equation*}
c_1|x|^{-\gamma}\le u_0(x)\le c_2 |x|^{-\gamma}\quad\forall x\in\R^n\setminus\{0\}
\end{equation*}
for some constants $c_1>0$, $c_2>0$ and $\frac{2}{1-m}<\gamma<\frac{n-2}{m}$ was proved by K. M. Hui and Soojung Kim in \cite{HKs}. Asymptotic large time behaviour of such solution was also proved by K. M. Hui and Soojung Kim in \cite{HKs} when $\frac{2}{1-m}<\gamma<n$. 

Let $a_1,a_2,\dots,a_{i_0}\in\Omega$, $\widehat{\Omega}=\Omega\setminus\{a_1,a_2,\dots,a_{i_0}\}$ and $\widehat{\R^n}=\R^n\setminus\{a_1,a_2,\dots,a_{i_0}\}$.  For any $\delta>0$, let $\Omega_{\delta}=\Omega\bs\left(\cup_{i=1}^{i_0}B_{\delta}(a_i)\right)$ and  $\R^n_{\delta}=\R^n\bs\left(\cup_{i=1}^{i_0}B_{\delta}(a_i)\right)$ where $B_R(x_0)=\{x\in\R^n:|x-x_0|<R\}$, $B_R=B_R(0)$, $\widehat{B}_R(x_0)=B_R(x_0)\bs\left\{x_0\right\}$ and $\widehat{B}_R=\widehat{B}_R(0)$ for any $x_0\in\R^n$ and $R>0$. Let $\delta_0(\Omega)=\frac{1}{3}\min_{1\leq i,j\leq i_0}\left(\textbf{dist}(a_i,\Omega),\left|a_i-a_j\right|\right)$ and $\delta_0(\R^n)=\frac{1}{3}\min_{1\leq i,j\leq i_0}|a_i-a_j|$. For any $0<\delta\le\delta_0(\Omega)$, let $D_{\delta}=\{x\in\Omega:\mbox{ dist}\,(x,\1\Omega)<\delta\}$. Let $R_0>0$ be such that  $a_1,\cdots,a_{i_0}\in B_{R_0}$. For any $R>R_0$ and $0<\delta\le\delta_0(\R^n)$,
let $\Omega_{\delta,R}=B_R\setminus \cup_{i=1}^{i_0}B_{\delta}(a_i)$.
When there is no ambiguity we will drop the parameter and write $\delta_0$ instead of $\delta_0(\Omega)$ or $\delta_0(\R^n)$. 
Unless stated otherwise we will assume that $0<m<\frac{n-2}{n}$ and $n\ge 3$ for the rest of the paper.

Existence of singular solutions of \eqref{fde} in $\widehat{\Omega}\times (0,T)$ which blows up at $\{a_1,a_2,\dots,a_{i_0}\}\times (0,T)$ was proved by K.M.~Hui and Sunghoon Kim in \cite{HK1} when  the initial value $u_0$ satisfies 
\begin{equation*}
u_0(x)\approx |x-a_i|^{-\gamma_i}\quad\mbox{ for }x\approx a_i\quad\forall i=1,2,\dots,i_0
\end{equation*} 
for some constants $\gamma_i>\max\left(\frac{n}{2m},\frac{n-2}{m}\right)$ for any $i=1,2,\dots,i_0$.  When  $0\le f\in L^{\infty}(\partial\Omega\times[0,\infty))$ and the initial value $0\leq u_0\in L_{loc}^p(\2{\Omega}\setminus\{a_1,\cdots,a_{i_0}\})$ ($L^p_{loc}(\widehat{\R^n})$ respectively) for some constant  $p>\frac{n(1-m)}{2}$  satisfies 
\begin{equation}\label{u0-lower-blow-up-rate}
u_0(x)\geq \frac{\lambda_i}{|x-a_i|^{\gamma_i}} \qquad \forall 0<|x-a_i|<\delta_1,\,\,i=1, \cdots, i_0
\end{equation}
for some constants $0<\delta_1<\min (1,\delta_0)$, $\lambda_1$, $\cdots$, $\lambda_{i_0}\in\R^+$ and $\gamma_1,\cdots,\gamma_{i_0}\in \left(\frac{2}{1-m},\infty\right)$, existence of singular solutions of 
\begin{equation}\label{Dirichlet-blow-up-problem}
\begin{cases}
\begin{aligned}
u_t&=\La u^m \quad \mbox{in $\widehat{\Omega}\times(0,\infty)$}\\
u&=f \qquad \mbox{ on $\partial\Omega\times(0,\infty)$}\\
u(a_i,t)&=\infty \qquad \forall t>0, i=1,\cdots,i_0\\
u(x,0)&=u_0(x) \quad \mbox{in $\widehat{\Omega}$}
\end{aligned}
\end{cases}
\end{equation}
and 
\begin{equation}\label{cauchy-blow-up-problem}
\begin{cases}
\begin{aligned}
u_t&=\La u^m \quad \mbox{in $\widehat{\R^n}\times(0,\infty)$}\\
u(a_i,t)&=\infty \qquad \forall i=1,\cdots,i_0,\,\,t>0\\
u(x,0)&=u_0(x) \quad \mbox{in $\widehat{\R^n}$}
\end{aligned}
\end{cases}
\end{equation}
respectively were proved by K.M.~Hui and Sunghoon Kim in \cite{HK2}. It was proved in \cite{HK2} that the singular solutions of \eqref{Dirichlet-blow-up-problem} and \eqref{cauchy-blow-up-problem} constructed in \cite{HK2} have the property  that for any $T>0$ and $\delta_2\in (0,\delta_1)$ there exists a constant $C_1>0$  such that
\begin{equation}\label{eq-lower-limit-of-u-near-blow-up-points}
u(x,t) \ge\frac{C_1}{|x-a_i|^{\gamma_i}}\quad \forall 0<|x-a_i|<\delta_2, 0<t<T, i=1,2,\cdots,i_0.
\end{equation}
Moreover \cite{HK2} if the initial value $u_0$ also satisfies
\begin{equation}\label{upper-blow-up-rate-initial-data}
u_0(x)\le\frac{\lambda_i'}{|x-a_i|^{\gamma_i'}} \qquad \forall 0<|x-a_i|<\delta_1,i=1, \cdots, i_0,
\end{equation}
for some constants $\lambda_1'$, $\cdots$, $\lambda_{i_0}'\in\R^+$, and $\gamma_i'\ge\gamma_i$ for all $i=1,\cdots,i_0$, then for any $T>0$ and $\delta_2\in (0,\delta_1)$ there exists a constant $C_2>0$  such that the singular solutions of \eqref{Dirichlet-blow-up-problem} and \eqref{cauchy-blow-up-problem} constructed in \cite{HK2} satisfies 
\begin{equation}\label{u-upper-blow-up-rate}
u(x,t)\le\frac{C_2}{|x-a_i|^{\gamma_i'}} \quad \forall 0<|x-a_i|<\delta_2, 0<t<T, i=1,2,\cdots,i_0.
\end{equation}
When $f\ge\mu_0$ and $u_0\ge\mu_0$ for some constant $\mu_0>0$, the singular solutions of \eqref{Dirichlet-blow-up-problem} and \eqref{cauchy-blow-up-problem} constructed in \cite{HK2} also satisfy 
\begin{equation}\label{u-lower-bd5}
u(x,t)\ge\mu_0\quad\forall x\in \widehat{\Omega}\,\,\, (\widehat{\R^n} \mbox{ respectively}), t>0.
\end{equation}
Asymptotic large time behaviour of such singular solutions was also proved by K.M.~Hui and Sunghoon Kim in \cite{HK2}.

In this paper we will extend the results of J.L. Vazquez and M. Winkler \cite{VW2} and prove the uniqueness of singular solutions of  \eqref{Dirichlet-blow-up-problem} and \eqref{cauchy-blow-up-problem}. We will also construct initial data $u_0$ such that the corresponding singular solution of \eqref{Dirichlet-blow-up-problem} with $f=\mu_0>0$ oscillates between infinity and some positive constant as $t\to\infty$. More precisely we will prove the following results.

\begin{thm}\label{uniqueness-thm-bd-domain}
Let $n\geq 3$, $0<m<\frac{n-2}{n}$, $0<\delta_1<\min (1,\delta_0)$, $\mu_0>0$, $f_1, f_2\in C^3(\partial\Omega\times (0,\infty))\cap L^{\infty}(\partial\Omega\times (0,\infty))$ be such that $f_2\ge f_1\ge\mu_0$ on $\partial\Omega\times (0,\infty)$ and 
\begin{equation}\label{u0-1-2-ineqn}
\mu_0\leq u_{0,1}\le u_{0,2}\in L_{loc}^p(\2{\Omega}\setminus\{a_1,\cdots,a_{i_0}\})\quad\mbox{ for some constant  }p>\frac{n(1-m)}{2}
\end{equation}
be such that 
\begin{equation}\label{2initial-value-lower-upper-bd}
\frac{\lambda_i}{|x-a_i|^{\gamma_i}}\le u_{0,1}(x)\le u_{0,2}\le \frac{\lambda_i'}{|x-a_i|^{\gamma_i'}} \qquad \forall 0<|x-a_i|<\delta_1,i=1, \cdots, i_0
\end{equation}
holds for some constants $\lambda_1$, $\cdots$, $\lambda_{i_0}$, $\lambda_1'$, $\cdots$, $\lambda_{i_0}'\in\R^+$ and 
\begin{equation}\label{gamma-gamma'-lower-bd}
\gamma_i'\ge\gamma_i>\frac{2}{1-m}\quad\forall i=1,2,\dots,i_0.
\end{equation}  
Suppose $u_1$, $u_2$, are the solutions of \eqref{Dirichlet-blow-up-problem} with $u_0=u_{0,1}, u_{0,2}$, $f=f_1, f_2$, respectively which satisfies
\begin{equation}\label{u1-2-lower-bd}
u_j(x,t)\ge\mu_0\quad\forall x\in \widehat{\Omega}, t>0, j=1,2
\end{equation}
such that for any constants $T>0$ and $\delta_2\in (0,\delta_1)$ there exist constants $C_1=C_1(T)>0$, $C_2=C_2(T)>0$, such that
\begin{equation}\label{u-lower-upper-blow-up-rate}
\frac{C_1}{|x-a_i|^{\gamma_i}}\le u_j(x,t)\le\frac{C_2}{|x-a_i|^{\gamma_i'}} \quad \forall 0<|x-a_i|<\delta_2, 0<t<T, i=1,2,\dots,i_0, j=1,2.
\end{equation}
Suppose $u_1$, $u_2$, also satisfies
\begin{equation}\label{initial-value-l1-sense}
\|u_i(\cdot,t)-u_{0,i}\|_{L^1(\Omega_{\delta})}\to 0\quad\mbox{ as }t\to 0\quad\forall 0<\delta<\delta_0, i=1,2.
\end{equation}
Then 
\begin{equation}\label{u1-u2-compare}
u_1(x,t)\le u_2(x,t)\quad\forall x\in \widehat{\Omega}, t>0.
\end{equation}
\end{thm}

\begin{thm}\label{uniqueness-thm2-bd-domain}
Let $n\geq 3$, $0<m<\frac{n-2}{n}$, $0<\delta_1<\min (1,\delta_0)$, $\mu_0>0$, $\mu_0\le f_1\le f_2\in L^{\infty}(\partial\Omega\times (0,\infty))$ and \eqref{u0-1-2-ineqn},  \eqref{2initial-value-lower-upper-bd},
hold for some constants $\lambda_1$, $\cdots$, $\lambda_{i_0}$, $\lambda_1'$, $\cdots$, $\lambda_{i_0}'\in\R^+$ and 
\begin{equation}\label{gamma-gamma'-lower-bd2}
\frac{2}{1-m}<\gamma_i\le\gamma_i'<\frac{n-2}{m}\quad\forall i=1,2,\dots,i_0.
\end{equation}  
Suppose $u_1$, $u_2$, are the solutions of \eqref{Dirichlet-blow-up-problem} with $u_0=u_{0,1}, u_{0,2}$, $f=f_1, f_2$, respectively which 
satisfies  \eqref{u1-2-lower-bd} and \eqref{initial-value-l1-sense} 
such that for any constants $T>0$ and $\delta_2\in (0,\delta_1)$ there exist constants $C_1=C_1(T)>0$, $C_2=C_2(T)>0$, such that \eqref{u-lower-upper-blow-up-rate} holds.
Then \eqref{u1-u2-compare} holds.
\end{thm}

\begin{thm}\label{uniqueness-thm-Rn}
Let $n\geq 3$, $0<m<\frac{n-2}{n}$, $0<\delta_1<\min (1,\delta_0)$, $\mu_0>0$, $R_1>R_0$ and $\mu_0\leq u_{0,1}\le u_{0,2}\in L_{loc}^p(\widehat{\R^n})$ for some constant  $p>\frac{n(1-m)}{2}$ such that
\begin{equation}\label{u0j-mu0-l1-difference}
\int_{\R^n\setminus B_{R_1}}|u_{0,j}-\mu_0|\,dx<\infty\quad\forall j=1,2.
\end{equation}
Let \eqref{2initial-value-lower-upper-bd} hold for some constants $\lambda_1$, $\cdots$, $\lambda_{i_0}$, $\lambda_1'$, $\cdots$, $\lambda_{i_0}'\in\R^+$ and
\begin{equation}\label{gamma-upper-lower-bd1}
\frac{2}{1-m}<\gamma_i\le\gamma_i'<n\quad\forall i=1,2,\dots,i_0.
\end{equation}
Suppose $u_1$, $u_2$, are the solutions of \eqref{cauchy-blow-up-problem} with $u_0=u_{0,1}, u_{0,2}$ respectively which satisfy
\begin{equation}\label{u1-2-Rn-lower-bd}
u_j(x,t)\ge\mu_0\quad\forall x\in \widehat{\R^n}, t>0, j=1,2
\end{equation}
and
\begin{equation}\label{uj-mu0-l1-difference}
\int_{\widehat{\R^n}}|u_j(x,t)-\mu_0|\,dx\le\int_{\widehat{\R^n}}|u_{0,j}-\mu_0|\,dx\quad\forall t>0, j=1,2
\end{equation}
such that for any constants $T>0$ and $\delta_2\in (0,\delta_1)$ there exist constants $C_1=C_1(T)>0$, $C_2=C_2(T)>0$, such that \eqref{u-lower-upper-blow-up-rate} holds. Then 
\begin{equation}\label{u1-u2-Rn-compare}
u_1(x,t)\le u_2(x,t)\quad\forall x\in \widehat{\R^n}, t>0.
\end{equation}
\end{thm}

\begin{thm}\label{existence-oscillating-soln-thm}
Let $n\geq 3$, $0<m<\frac{n-2}{n}$, $0<\delta_1<\min (1,\delta_0)$ and $\mu_0>0$. There there exists $u_0\in L_{loc}^p(\2{\Omega}\setminus\{a_1,\cdots,a_{i_0}\})$ for some constant $p>\frac{n(1-m)}{2}$, $u_0\ge\mu_0$ in $\widehat{\Omega}$, such that 
\begin{equation}\label{initial-value-lower-upper-bd}
\frac{\lambda_i}{|x-a_i|^{\gamma_i}}\le u_0(x)\le \frac{\lambda_i'}{|x-a_i|^{\gamma_i'}} \qquad \forall 0<|x-a_i|<\delta_1,i=1, \cdots, i_0
\end{equation}
for some constants  satisfying \eqref{gamma-gamma'-lower-bd} and $\lambda_1$, $\cdots$, $\lambda_{i_0}$, $\lambda_1'$, $\cdots$, $\lambda_{i_0}'\in\R^+$  such that 
\begin{equation}\label{Dirichlet-blow-up-problem-constant-bdary-value}
\begin{cases}
\begin{aligned}
u_t&=\La u^m \quad \mbox{in $\widehat{\Omega}\times(0,\infty)$}\\
u&=\mu_0\qquad \mbox{ on $\partial\Omega\times(0,\infty)$}\\
u(a_i,t)&=\infty \qquad \forall t>0, i=1,\cdots,i_0\\
u(x,0)&=u_0(x) \quad \mbox{in $\widehat{\Omega}$}
\end{aligned}
\end{cases}
\end{equation} has a unique solution $u$ with the property that $u$ oscillates between $\mu_0$ and infinity as $t\to\infty$.
\end{thm}

The plan of the paper is as follows. In section 2 we will prove the uniqueness of singular solutions of \eqref{Dirichlet-blow-up-problem} and \eqref{cauchy-blow-up-problem}. In section 3 we will prove the existence of initial data such that the corresponding solution of \eqref{Dirichlet-blow-up-problem} with $f=\mu_0>0$ oscillates between infinity and some positive constant as $t\to\infty$.

We start with some definitions. For any $0\leq f\in L^{\infty}(\partial\Omega\times(0,\infty))$ and $0\le u_0\in L^1_{loc}(\widehat{\Omega})$, we say that $u$ is a solution of 
\begin{equation}\label{fde-Dirichlet-blow-up-problem}
\begin{cases}
\begin{aligned}
u_t&=\La u^m \qquad \mbox{in $\widehat{\Omega}\times(0,\infty)$}\\
u&=f \qquad\quad \mbox{on $\partial\Omega\times(0,\infty)$}\\
u(x,0)&=u_0(x) \quad\,\,\, \mbox{in $\widehat{\Omega}$}
\end{aligned}
\end{cases}
\end{equation}
if $u\in L_{loc}^{\infty}(\2{\Omega}\setminus\{a_1,\cdots,i_0\}\times (0,\infty))$ is positive in $\widehat{\Omega}\times(0,\infty)$ and satisfies \eqref{fde} in $\widehat{\Omega}\times(0,\infty)$  in the classical sense with
\begin{equation}\label{u-initial-value}
\|u(\cdot,t)-u_0\|_{L^1(K)}\to 0\quad \mbox{ as }t\to 0 
\end{equation}
for any compact set $K\subset\widehat{\Omega}$ and
\begin{equation}\label{eq-formula-for-weak-sols-of-main-problem}
\begin{aligned}
&\int_{t_1}^{t_2}\int_{\widehat{\Omega}}\left(u\eta_t+u^m\La\eta\right)\,dxdt\\
&\qquad \qquad =\int_{t_1}^{t_2}\int_{\partial\Omega}f^m\frac{\partial\eta}{\partial\nu}\,d\sigma dt+\int_{\widehat{\Omega}}u(x,t_2)\eta(x,t_2)\,dx-\int_{\widehat{\Omega}}u(x,t_1)\eta(x,t_1)\,dx
\end{aligned}
\end{equation}
 for any $t_2>t_1>0$ and $\eta\in C_c^2((\overline{\Omega}\bs\left\{a_1,\cdots,a_{i_0}\right\})\times(0,\infty))$ satisfying $\eta\equiv 0$ on $\partial\Omega\times(0,T)$.  We say that $u$ is a solution  of \eqref{Dirichlet-blow-up-problem} if $u$ is a solution of \eqref{fde-Dirichlet-blow-up-problem} and satisfies 
\begin{equation}\label{u-blow-up-infty}
u(x,t)\to\infty\quad\mbox{ as }x\to a_i\quad\forall t>0,i=1,\cdots,i_0. 
\end{equation}

For any $0\le u_0\in L^1_{loc}(\widehat{\R^n})$ we say that $u$ is a solution  of \eqref{cauchy-blow-up-problem} if $u\in L_{loc}^{\infty}(\widehat{\R^n}\times(0,\infty))$ is positive in $\widehat{\R^n}\times(0,\infty)$ and satisfies \eqref{fde} in $\widehat{\R^n}\times(0,\infty)$  in the classical sense and \eqref{u-initial-value}, \eqref{u-blow-up-infty},  hold for any compact set $K\subset\widehat{\R^n}$. 

For any set $A\subset\R^n$, we let $\chi_A$ be the characteristic function of the set $A$. For any $a\in\R$, we let $a_+=\max (0,a)$.

\section{Uniqueness of solution}
\setcounter{equation}{0}
\setcounter{thm}{0} 

In this section we will prove the uniqueness of singular solutions of \eqref{Dirichlet-blow-up-problem} and \eqref{cauchy-blow-up-problem}.

\begin{lem}\label{u-local-l1-bd-intial-value-lem}
Let $n\geq 3$, $0<m<\frac{n-2}{n}$, $0<\delta_1<\delta_0$, $0\le f\in L^{\infty}(\partial\Omega\times[0,\infty))$ and $0\leq u_0\in L_{loc}^p(\2{\Omega}\setminus\{a_1,\cdots,a_{i_0}\})$ for some constant  $p>\frac{n(1-m)}{2}$ be such that \eqref{u0-lower-blow-up-rate}
holds for some constants $\lambda_1$, $\cdots$, $\lambda_{i_0}\in\R^+$ and $\gamma_1,\cdots,\gamma_{i_0}\in \left(\frac{2}{1-m},\infty\right)$. Let $u$ be the  solution of \eqref{Dirichlet-blow-up-problem} constructed in Theorem 1.1 of \cite{HK2}. Then there exists a constant $C>0$ such that 
\begin{equation}\label{u-integral-bd-by-time-0-integral}
\int_{D_{\delta}}u(x,t)\,dx
\le\left\{\left(\int_{D_{2\delta}}u_0\,dx\right)^{1-m}+Ct\right\}^{\frac{1}{1-m}}+|D_{\delta}|\|f\|_{L^{\infty}}\quad\forall t>0, 0<\delta<\delta_0/2
\end{equation}
holds. 
\end{lem}
\begin{proof}
We will use a modification of the proof of Theorem 2.2 of \cite{HP} to prove this lemma.  For any $0<\3<1$, let  $k>\|f\|_{L^{\infty}}+\3$. Let $0<\delta<\delta_0/2$. We choose $\phi\in C^{\infty}(\2{\Omega})$ such that $0\le\phi\le 1$ in
$\2{\Omega}$, $\phi (x)=0$ for any $x\in\Omega\setminus D_{2\delta}$ and $\phi (x)=1$ for any $x\in \2{D_{\delta}}$. Let $\alpha>\frac{2}{1-m}$ and $\eta(x)=\phi(x)^{\alpha}$. Then by direct computation,
\begin{equation*}
C_{\eta}:=\left(\int_{\Omega}(\eta^{-m}|\Delta\eta|)^{\frac{1}{1-m}}\,dx\right)^{1-m}<\infty.
\end{equation*}
For any $0<\3<1$ and $M>0$, let
\begin{equation*}
\left\{\begin{aligned}
u_{0,M}(x)=&\min(u_0(x),M)\\
u_{0,\3,M}(x)=&\min(u_0(x),M)+\3\\
f_{\3}(x,t)=&f(x,t)+\3\quad \forall (x,t)\in\partial\Omega\times(0,\infty)  
\end{aligned}\right.
\end{equation*}
and let $u_M$ and $u_{\3,M}$ be the solutions of 
\begin{equation*}
\begin{cases}
\begin{aligned}
u_t&=\La u^m \qquad  \mbox{in $\Omega\times(0,\infty)$}\\
u(x,t)&=f_{\3} \qquad\,\,\, \mbox{ on $\partial\Omega\times(0,\infty)$}\\
u(x,0)&=u_{0,\3,M}  \quad \mbox{ in $\Omega$}.
\end{aligned}
\end{cases}
\end{equation*}
and
\begin{equation*}
\begin{cases}
\begin{aligned}
u_t&=\La u^m \quad\mbox{ in }\Omega\times (0,T_M)\\
u(x,t)&=f \quad\quad \mbox{ on $\partial\Omega\times(0,T_M)$}\\
u(x,0)&=u_{0,M} \quad \mbox{in $\Omega$}
\end{aligned}
\end{cases}
\end{equation*}
respectively constructed in \cite{HK2} for some maximal time $T_M>0$ of existence. By the result in \cite{HK2} $T_M\to\infty$ as $M\to\infty$. Moreover $u_{\3,M}$ decreases and converges to $u_M$ in $\Omega\times (0,T_M)$ uniformly in $C^{2,1}(K)$ for every compact subset $K$ of $\Omega$ as $\3\to 0$ and $u_M$ increases and converges to $u$ in $\widehat{\Omega}\times (0,\infty)$ uniformly in $C^{2,1}(K)$ for every compact subset $K$ of $\widehat{\Omega}$ as $M\to\infty$. By approximation we may assume without loss of generality that $u_{\3,M}\in C^2(\2{\Omega}\times (0,\infty))$. 
Then by the Kato inequality (\cite{DK}, \cite{K}),
\begin{align}\label{u-integral-ineqn}
\frac{\1}{\1 t}\left(\int_{\Omega}(u_{\3,M}-k)_+\eta\,dx\right)\le&\int_{\Omega}(u_{\3,M}^m-k^m)_+\Delta\eta\,dx\notag\\
\le&C\int_{\Omega}(u_{\3,M}-k)_+^m|\Delta\eta|\,dx\notag\\
\le&C\left(\int_{\Omega}(u_{\3,M}-k)_+\eta\,dx\right)^m\left(\int_{\Omega}(\eta^{-m}|\Delta\eta|)^{\frac{1}{1-m}}\,dx\right)^{1-m}\notag\\
=&C_1\left(\int_{\Omega}(u_{\3,M}-k)_+\eta\,dx\right)^m\quad\forall t>0, 0<\3<1, M>0
\end{align}
for some constants $C>0$, $C_1>0$. Integrating \eqref{u-integral-ineqn} over $(0,t)$ and letting $\3\to 0$, $M\to\infty$ and $k\to \|f\|_{L^{\infty}}$,
\begin{equation*}
\int_{D_{\delta}}(u(x,t)-\|f\|_{L^{\infty}})_+\,dx
\le\left\{\left(\int_{D_{2\delta}}(u_0-\|f\|_{L^{\infty}})_+\,dx\right)^{1-m}+C_1(1-m)t\right\}^{\frac{1}{1-m}}\quad\forall t>0
\end{equation*}
and \eqref{u-integral-bd-by-time-0-integral} follows.
\end{proof}

\begin{prop}\label{u-go-to-u0-l1-5}
Let $n\geq 3$, $0<m<\frac{n-2}{n}$, $0<\delta_1<\delta_0$, $0\le f\in L^{\infty}(\partial\Omega\times[0,\infty))$ and $0\leq u_0\in L_{loc}^p(\2{\Omega}\setminus\{a_1,\cdots,a_{i_0}\})$ for some constant  $p>\frac{n(1-m)}{2}$ be such that \eqref{u0-lower-blow-up-rate}
holds for some constants $\lambda_1$, $\cdots$, $\lambda_{i_0}\in\R^+$ and $\gamma_1,\cdots,\gamma_{i_0}\in \left(\frac{2}{1-m},\infty\right)$. Let $u$ be the  solution of \eqref{Dirichlet-blow-up-problem} constructed in Theorem 1.1 of \cite{HK2}. Then 
\begin{equation}\label{initial-value-l1-sense2}
\|u(\cdot,t)-u_0\|_{L^1(\Omega_{\delta})}\to 0\quad\mbox{ as }t\to 0\quad\forall 0<\delta<\delta_0.
\end{equation} 
\end{prop}
\begin{proof}
Let $0<\delta<\delta_0$ and $0<\delta'<\delta_0/2$. Then by Lemma \ref{u-local-l1-bd-intial-value-lem},
\begin{align}\label{u-l1-t-go-to-0}
\|u(\cdot,t)-u_0\|_{L^1(\Omega_{\delta})}
\le&\|u(\cdot,t)-u_0\|_{L^1(\Omega_{\delta}\setminus D_{\delta'})}
+\|u(\cdot,t)\|_{L^1(D_{\delta'})}+\|u_0\|_{L^1(D_{\delta'})}\notag\\
\le&\|u(\cdot,t)-u_0\|_{L^1(\Omega_{\delta}\setminus D_{\delta'})}
+\left(\|u_0\|_{L^1(D_{2\delta'})}^{1-m}+Ct\right)^{\frac{1}{1-m}}+|D_{\delta'}|\|f\|_{L^{\infty}}+\|u_0\|_{L^1(D_{\delta'})}.
\end{align}
Letting $t\to 0$ in \eqref{u-l1-t-go-to-0},
\begin{align*}
&\limsup_{t\to 0}\|u(\cdot,t)-u_0\|_{L^1(\Omega_{\delta})}
\le |D_{\delta'}|\|f\|_{L^{\infty}}+2\|u_0\|_{L^1(D_{2\delta'})}\quad\forall
0<\delta'<\delta_0/2\notag\\
\Rightarrow\quad&\lim_{t\to 0}\|u(\cdot,t)-u_0\|_{L^1(\Omega_{\delta})}=0\quad\mbox{ as }\delta'\to 0
\end{align*}
and \eqref{initial-value-l1-sense2} follows.
\end{proof}

\begin{proof}[\textbf{Proof of Theorem \ref{uniqueness-thm-bd-domain}}:]
We will use a modification of the proof of Theorem 6 of \cite{VW2} to prove the theorem.
Let $0<\delta<\delta_0$ and $t_1>t_0>0$. Since $u_1, u_2\in L_{loc}^{\infty}(\2{\Omega}\setminus\{a_1,\cdots,i_0\}\times (0,\infty))$ there exists a constant $M_1>0$ such that
\begin{equation}\label{u1-2-local-upper-bd}
u_j(x,t)\le M_1\quad\forall x\in\Omega_{\delta},t_0\le t\le t_1,j=1,2.
\end{equation}
By \eqref{u1-2-lower-bd} and \eqref{u1-2-local-upper-bd}, the equation \eqref{fde} for $u_1$ and $u_2$ are uniformly parabolic on every compact subset of $\2{\Omega}\setminus\{a_1,\cdots,i_0\}\times (0,\infty)$. Hence by the parabolic Schauder estimates \cite{LSU}, $u_1, u_2\in C^{2,1}(\2{\Omega}\setminus\{a_1,\cdots,i_0\}\times (0,\infty))$. 

We choose a nonnegative monotone increasing function $\phi\in C^{\infty}(\R)$ such that $\phi(s)=0$ for any $s\le 1/2$ and $\phi(s)=1$ for any $s\ge 1$. For any $0<\delta<\delta_0$, let $\phi_{\delta}(x)=\phi (|x|/\delta)$. Then $|\nabla\phi_{\delta}|\le C/\delta$ and $|\Delta\phi_{\delta}|\le C/\delta^2$. Let $\alpha>\max (2+n,\gamma_1,\gamma_2,\cdots,\gamma_{i_0})-n$. We choose $0<\psi\in C^{\infty}(\2{\Omega}\setminus\{a_1,\cdots, a_{i_0}\})$ such that 
$\psi (x)=|x-a_i|^{\alpha}$ for any $x\in\cup_{i=1}^{i_0}B_{\delta_0}(a_i)$. Let $\delta_2=\delta_1/2$ and $T>0$. Then there exists a constant $c_1>0$ such that 
\begin{equation}\label{psi+ve}
\psi (x)\ge c_1\quad\forall x\in \2{\Omega}\setminus \cup_{i=1}^{i_0}B_{\delta_2}(a_i).
\end{equation} 
By \eqref{u-lower-upper-blow-up-rate} and the choice of $\alpha$, for any $i=1,\cdots,i_0$, 
\begin{align}\label{u1-u2-blow-neighd-l1-bd}
\int_{B_{\delta_2}(a_i)}|x-a_i|^{\alpha}(u_1-u_2)_+(x,t)\,dx\le&C_T\int_0^{\delta_2}\rho^{\alpha+n-\gamma_i-1}\,d\rho=C_T'\delta_2^{\alpha+n-\gamma_i}<\infty\quad\forall 0<t<T
\end{align}
for some constants $C_T>0$, $C_T'>0$. Since $u_1,u_2\in L_{loc}^{\infty}(\2{\Omega}\setminus\{a_1,\cdots,i_0\}\times (0,\infty))$, by \eqref{initial-value-l1-sense} and \eqref{u1-u2-blow-neighd-l1-bd} for any $T>0$ there exists a constant $C_0(T)>0$ such that
\begin{equation}\label{u1-u2-l1-with-psi-bd}
\int_{\widehat{\Omega}}\psi (x)(u_1-u_2)_+(x,t)\,dx\le C_0(T)<\infty\quad\forall 0<t<T.
\end{equation}
Let
\begin{equation*}
w_{\delta}(x)=\Pi_{i=1}^{i_0}\phi_{\delta}(x-a_i).
\end{equation*}
By the Kato inequality (\cite{DK}, \cite{K}) for any $0<\delta<\delta_1$, $t>0$,
\begin{align}\label{u1-2-difference-in-l1-time-derivative-ineqn}
&\frac{\1}{\1 t}\left(\int_{\widehat{\Omega}}(u_1-u_2)_+\psi w_{\delta}\,dx\right)\notag\\
\le&\int_{\widehat{\Omega}}(u_1^m-u_2^m)_+\Delta (\psi w_{\delta})\,dx\notag\\
=&\int_{\widehat{\Omega}}\left\{w_{\delta}\Delta \psi+2\nabla\psi\cdot\nabla w_{\delta}+\psi\Delta w_{\delta}\right\}(u_1^m-u_2^m)_+\,dx\notag\\
\le&\int_{\widehat{\Omega}}(u_1^m-u_2^m)_+w_{\delta}\Delta \psi\,dx+\frac{C}{\delta}\sum_{i=1}^{i_0}\int_{\delta/2\le |x-a_i|\le \delta}|x-a_i|^{\alpha-1}(u_1^m-u_2^m)_+(x,t)\,dx\notag\\
&\qquad +\frac{C}{\delta^2}\sum_{i=1}^{i_0}\int_{\delta/2\le |x-a_i|\le \delta}|x-a_i|^{\alpha}(u_1^m-u_2^m)_+(x,t)\,dx\notag\\
\le&\int_{\widehat{\Omega}}(u_1^m-u_2^m)_+w_{\delta}\Delta \psi\,dx+C\sum_{i=1}^{i_0}\int_{\delta/2\le |x-a_i|\le \delta}|x-a_i|^{\alpha-2}(u_1^m-u_2^m)_+(x,t)\,dx.
\end{align}
By direct computation,
\begin{equation}\label{laplacian-norm-power}
\Delta\psi=\Delta |x-a_i|^{\alpha}=\alpha (\alpha +n-2)|x-a_i|^{\alpha-2}\quad\forall x\in\widehat{B_{\delta_0}}(a_i),i=1,\cdots,i_0.
\end{equation}
By \eqref{u1-2-lower-bd} and the mean value theorem,
\begin{equation}\label{u1-u2-difference0}
(u_1^m-u_2^m)_+(x,t)\le m\mu_0^{m-1}(u_1-u_2)_+(x,t)\quad\forall x\in \widehat{\Omega}, t>0.
\end{equation}
By \eqref{psi+ve}, \eqref{u1-2-difference-in-l1-time-derivative-ineqn}, \eqref{laplacian-norm-power} and \eqref{u1-u2-difference0},
\begin{align}\label{u1-2-difference-in-l1-time-derivative-ineqn2}
&\frac{\1}{\1 t}\left(\int_{\widehat{\Omega}}(u_1-u_2)_+\psi w_{\delta}\,dx\right)\notag\\
\le&C\int_{\Omega\setminus\cup_{i=1}^{i_0}B_{\delta_2}(a_i)}(u_1^m-u_2^m)_+(x,t)\,dx
+C\int_{\cup_{i=1}^{i_0}B_{\delta_2}(a_i)}|x-a_i|^{\alpha-2}(u_1^m-u_2^m)_+(x,t)\,dx\notag\\
\le&C\int_{\widehat{\Omega}}(u_1-u_2)_+(x,t)\psi(x)\,dx
+C\int_{\cup_{i=1}^{i_0}B_{\delta_2}(a_i)}|x-a_i|^{\alpha-2}(u_1^m-u_2^m)_+(x,t)\,dx.
\end{align}
By \eqref{u-lower-upper-blow-up-rate} and the mean value theorem for any  $|x-a_i|\le\delta_2$, $0<t<T$, $i=1,\cdots,i_0$,
\begin{align}\label{u1-u2-difference-local-ineqn}
|x-a_i|^{\alpha-2}(u_1^m-u_2^m)_+(x,t)\le &m|x-a_i|^{\alpha-2}u_2(x,t)^{m-1}(u_1-u_2)_+(x,t)\notag\\
\le &mC_1(T)^{m-1}|x-a_i|^{(1-m)\gamma_i-2+\alpha}(u_1-u_2)_+(x,t)\notag\\
\le &mC_1(T)^{m-1}\delta_0^{(1-m)\gamma_i-2}|x-a_i|^{\alpha}(u_1-u_2)_+(x,t)\notag\\
\le &mC_1(T)^{m-1}\delta_0^{(1-m)\gamma_i-2}\psi(x)(u_1-u_2)_+(x,t).
\end{align}
By \eqref{u1-2-difference-in-l1-time-derivative-ineqn2} and \eqref{u1-u2-difference-local-ineqn},
\begin{equation}\label{u1-2-difference-in-l1-time-derivative-ineqn3}
\frac{\1}{\1 t}\left(\int_{\widehat{\Omega}}(u_1-u_2)_+\psi w_{\delta}\,dx\right)
\le C_T\int_{\widehat{\Omega}}(u_1-u_2)_+(x,t)\psi(x)\,dx. 
\end{equation}
Integrating \eqref{u1-2-difference-in-l1-time-derivative-ineqn3} over $(0,t)$, by \eqref{initial-value-l1-sense} and \eqref{u1-u2-l1-with-psi-bd},
\begin{align}\label{u1-2-difference-in-l1-time-derivative-ineqn4}
&\int_{\widehat{\Omega}}(u_1-u_2)_+(x,t)\psi (x)w_{\delta}(x)\,dx\le 
C_T\int_0^t\int_{\widehat{\Omega}}(u_1-u_2)_+(x,t)\psi(x)\,dx\,dt\quad\forall 0<t<T\notag\\
\Rightarrow\quad&\int_{\widehat{\Omega}}(u_1-u_2)_+(x,t)\psi (x)\,dx\le 
C_T\int_0^t\int_{\widehat{\Omega}}(u_1-u_2)_+(x,t)\psi(x)\,dx\,dt\quad\forall 0<t<T\quad\mbox{ as }\delta\to 0.
\end{align}
By \eqref{u1-2-difference-in-l1-time-derivative-ineqn4}, 
\begin{equation}\label{u1<u2-bd-time}
u_1(x,t)\le u_2(x,t)\quad\forall x\in \widehat{\Omega}, 0<t<T.
\end{equation}
Letting $T\to\infty$ in \eqref{u1<u2-bd-time} we get \eqref{u1-u2-compare} and the theorem follows.
\end{proof}

By Theorem 1.1, Lemma 2.3, Lemma 2.15 of \cite{HK2} and Theorem \ref{uniqueness-thm-bd-domain} and Proposition \ref{u-go-to-u0-l1-5} we have the following result. 

\begin{thm}\label{unique-soln-thm-bd-domain}
Let $n\geq 3$, $0<m<\frac{n-2}{n}$, $0<\delta_1<\min (1,\delta_0)$, $\mu_0>0$, $f\in C^3(\partial\Omega\times (0,\infty))\cap L^{\infty}(\partial\Omega\times (0,\infty))$ be such that $f\ge\mu_0$ on $\partial\Omega\times (0,\infty)$ and $\mu_0\leq u_0\in L_{loc}^p(\2{\Omega}\setminus\{a_1,\cdots,a_{i_0}\})$ for some constant  $p>\frac{n(1-m)}{2}$ be such that \eqref{u0-lower-blow-up-rate} and \eqref{upper-blow-up-rate-initial-data} hold for some constants  satisfying \eqref{gamma-gamma'-lower-bd} and $\lambda_1$, $\cdots$, $\lambda_{i_0}$, $\lambda_1'$, $\cdots$, $\lambda_{i_0}'\in\R^+$. Then there exists a unique solution $u$ of \eqref{Dirichlet-blow-up-problem} 
which satisfies  \eqref{u-lower-bd5} and \eqref{initial-value-l1-sense2} such that for any constants $T>0$ and $\delta_2\in (0,\delta_1)$ there exist constants $C_1=C_1(T)>0$, $C_2=C_2(T)>0$, depending only on $\lambda_1$, 
$\cdots$, $\lambda_{i_0}$, $\lambda_1'$, $\cdots$, $\lambda_{i_0}'$, $\gamma_1$, $\cdots$, $\gamma_{i_0}$, 
$\gamma_1'$, $\cdots$, $\gamma_{i_0}'$, such that
\begin{equation}\label{u-lower-upper-blow-up-rate5}
\frac{C_1}{|x-a_i|^{\gamma_i}}\le u(x,t)\le\frac{C_2}{|x-a_i|^{\gamma_i'}} \quad \forall 0<|x-a_i|<\delta_2, 0<t<T, i=1,2,\dots,i_0
\end{equation}
holds.
\end{thm}

\begin{proof}[\textbf{Proof of Theorem \ref{uniqueness-thm2-bd-domain}}:]
We will use a modification of the proof of Lemma 2.3 of \cite{DaK} to prove the theorem.
Let
\begin{equation*}
A=A(x,t)=\left\{\begin{aligned}
&\frac{u_1(x,t)^m-u_2(x,t)^m}{u_1(x,t)-u_2(x,t)}\quad\forall x\in\widehat{\Omega}, t>0 \mbox{ satisfying }u_1(x,t)\ne u_2(x,t)\\
&mu_1(x,t)^{m-1}\qquad\qquad\forall x\in\widehat{\Omega}, t>0 \mbox{ satisfying }u_1(x,t)= u_2(x,t)\\
&0\qquad\qquad\qquad\qquad\,\,\forall x=a_i, i=1,\cdots, i_0,t>0.
\end{aligned}\right.
\end{equation*}
For any $k\in\Z^+$, let
\begin{equation*}
\alpha_k(x,t)=\left\{\begin{aligned}
&\frac{|u_1(x,t)^m-u_2(x,t)^m|}{|u_1(x,t)-u_2(x,t)|+(1/k)}\quad\forall x\in\widehat{\Omega}, t>0\\
&0\qquad\qquad\qquad\qquad\qquad\,\,\forall x=a_i, i=1,\cdots, i_0,t>0
\end{aligned}\right.
\end{equation*}
and $A_k=A_k(x,t)=\alpha_k(x,t)+k^{-1}$. We choose a nonnegative monotone increasing function $\phi\in C^{\infty}(\R)$ such that $\phi(s)=0$ for any $s\le 0$ and $\phi(s)=1$ for any $s\ge 1$. Let $0<\delta_2\le\delta_1/2$. For any  $\delta\in (0,\delta_2/2)$ and $j\ge 2/\delta_2$, let $\phi_j(x)=\phi(j(|x|-\delta))$. Let $t_1>t_0>0$ and $0\le h\in C^{\infty}_0(\Omega_{\delta_2})$. For any $k\in\Z^+$ and $0<\delta\le\delta_2/2$, let
$\psi_{k,\delta}$ be the solution of 
\begin{equation}
\left\{\begin{aligned}
&\psi_t+A_k\Delta\psi=0\quad\mbox{ in }\Omega_{\delta}\times (0,t_1)\\
&\psi(x,t)=0\qquad\quad\,\mbox{on }\1\Omega_{\delta}\times (0,t_1)\\
&\psi(x,t_0)=h(x)\quad\,\mbox{ in }\Omega_{\delta}
\end{aligned}\right.
\end{equation}
and 
\begin{equation*}
w_j(x)=\Pi_{i=1}^{i_0}\phi_j(x-a_i).
\end{equation*}
Then $|\nabla w_j|\le Cj$ and $|\Delta w_j|\le Cj^2$ for some constant $C>0$.
By the maximum principle, $0\le \psi_{k,\delta}\le \|h\|_{L^{\infty}}$ in 
$\Omega_{\delta}\times (0,t_1)$. Hence 
$\frac{\1\psi_{k,\delta}}{\1\nu}\le 0$ on $\1\Omega\times (0,t_1)$. 
Then
\begin{align}\label{u1-2-l1-ineqn1}
&\int_{\Omega_{\delta}}(u_1(x,t_1)-u_2(x,t_1))h(x)\,dx\notag\\
=&\int_{\Omega_{\delta}}(u_1(x,t_0)-u_2(x,t_0))\psi_{k,\delta}(x,t_0)w_j(x)\,dx+\int_{t_0}^{t_1}\int_{\1\Omega}(f_2^m-f_1^m)\frac{\1\psi_{k,\delta}}{\1\nu}\,d\sigma\, dt\notag\\
&\qquad +\int_{t_0}^{t_1}\int_{\Omega_{\delta}}(u_1-u_2)\left\{w_j\left(\1_t\psi_{k,\delta}+A\Delta \psi_{k,\delta}\right)+A\nabla w_j\cdot\nabla\psi_{k,\delta}+A\psi_{k,\delta}\Delta w_j\right\}\,dx\,dt\notag\\
\le&\|h\|_{L^{\infty}}\int_{\Omega_{\delta}}(u_1(x,t_0)-u_2(x,t_0))_+\,dx
+\int_{t_0}^{t_1}\int_{\Omega_{\delta}}|u_1-u_2||A-A_k||\Delta \psi_{k,\delta}|\,dx\,dt\notag\\
&\qquad +C\sum_{i=1}^{i_0}\int_{t_0}^{t_1}\int_{\delta\le |x-a_i|\le\delta+j^{-1}}|u_1^m-u_2^m|\left\{j|\nabla|x-a_i|\cdot\nabla\psi_{k,\delta}|+j^2\psi_{k,\delta}\right\}\,dx\,dt\notag\\
=&I_1+I_2+I_3.
\end{align}
We will now use a modification of the proof of Theorem 2.1 of \cite{PV} to estimate the derivative of $\psi_{k,\delta}$ on $\cup_{i=1}^{i_0}\1 B_{\delta}(a_i)\times (0,t_1)$. Let
\begin{equation}\label{qi-defn}
q_i(x)=\frac{\delta^{2-n}-|x-a_i|^{2-n}}{\delta^{2-n}-\delta_2^{2-n}}\cdot\|h\|_{L^{\infty}}\quad\forall i=1,\cdots,i_0.
\end{equation}
Then for any $i=1,\cdots,i_0$, $q_i$ satisfies
\begin{equation}\label{q-time-eqn}
\left\{\begin{aligned}
q_t+A_k\Delta q=&0\qquad\quad\mbox{ in }(B_{\delta_2}(a_i)\setminus\2{B_{\delta}(a_i)})\times (0,t_1)\\
q=&0\qquad\quad\mbox{ on }\1 B_{\delta}(a_i)\times (0,t_1)\\
q=&\|h\|_{L^{\infty}}\quad\mbox{ on }\1 B_{\delta_2}(a_i)\times (0,t_1)\\
q\ge&0 \qquad\quad\mbox{ on }B_{\delta_2}(a_i)\setminus\2{B_{\delta}(a_i)}
\end{aligned}\right.
\end{equation}
Since $\psi_{k,\delta}$ is a subsolution of \eqref{q-time-eqn}, by the maximum principle,
\begin{align}
&0\le \psi_{k,\delta}(x,t)\le q_i(x)\quad\forall\delta\le |x-a_i|\le\delta_2, 0<t\le t_1, i=1,\cdots,i_0\label{psi-k-delta-upper-bd}\\
\Rightarrow\quad&\left|\frac{\1\psi_{k,\delta}}{\1\nu}\right|\le\left|\frac{\1 q_i}{\1\nu}\right|
=\frac{(n-2)\delta^{1-n}}{\delta^{2-n}-\delta_2^{2-n}}\|h\|_{L^{\infty}}\quad\mbox{ on }\1 B_{\delta}(a_i)\times (0,t_1)\quad\forall i=1,\cdots,i_0.\label{psi-k-delta-derivative-upper-bd}
\end{align}
By \eqref{qi-defn} and the mean value theorem,
\begin{equation}\label{qi-upper-bd}
q_i(x)\le \frac{(n-2)j^{-1}\delta^{1-n}}{\delta^{2-n}-\delta_2^{2-n}}\|h\|_{L^{\infty}}\quad\forall \delta\le |x-a_i|\le\delta+j^{-1}, i=1,\cdots,i_0.
\end{equation}
By \eqref{u-lower-upper-blow-up-rate},  \eqref{gamma-gamma'-lower-bd2}, \eqref{u1-2-l1-ineqn1}, \eqref{psi-k-delta-upper-bd}, \eqref{psi-k-delta-derivative-upper-bd} and \eqref{qi-upper-bd},
\begin{align}\label{u1-2-l1-ineqn2}
I_3\le&C\sum_{i=1}^{i_0}j\int_{t_0}^{t_1}\int_{\delta\le |x-a_i|\le\delta+j^{-1}}|u_1^m-u_2^m|\left\{|\nabla|x-a_i|\cdot\nabla\psi_{k,\delta}|+\frac{(n-2)\delta^{1-n}}{\delta^{2-n}-\delta_2^{2-n}}\|h\|_{L^{\infty}}\right\}\,dx\,dt\notag\\
=&C\sum_{i=1}^{i_0}\int_{t_0}^{t_1}\int_{\1 B_{\delta}(a_i)}|u_1^m-u_2^m|\left\{\left|\frac{\1\psi_{k,\delta}}{\1\nu}\right|+\frac{(n-2)\delta^{1-n}}{\delta^{2-n}-\delta_2^{2-n}}\|h\|_{L^{\infty}}\right\}\,dx\,dt\quad\mbox{ as }j\to\infty\notag\\
\le&2C\frac{(n-2)\delta^{1-n}}{\delta^{2-n}-\delta_2^{2-n}}\|h\|_{L^{\infty}}\sum_{i=1}^{i_0}\int_{t_0}^{t_1}\int_{\1 B_{\delta}(a_i)}|u_1^m-u_2^m|\,d\sigma\,dt\notag\\
\le&\frac{C'\|h\|_{L^{\infty}}t_1\delta^{1-n}}{\delta^{2-n}-\delta_2^{2-n}}\sum_{i=1}^{i_0}\delta^{n-1-m\gamma_i'}.
\end{align}
By the same argument as the proof of Lemma 2.3 of \cite{DK},
\begin{equation}\label{I2-go-to-0}
\lim_{k\to\infty}I_2=0.
\end{equation}
Hence letting first $j\to\infty$ and then $k\to\infty$ in \eqref{u1-2-l1-ineqn1}, by \eqref{u1-2-l1-ineqn2} and \eqref{I2-go-to-0}, 
\begin{align}\label{u1-2-l1-ineqn4}
&\int_{\Omega_{\delta}}(u_1(x,t_1)-u_2(x,t_1))h(x)\,dx\notag\\
\le&\|h\|_{L^{\infty}}\int_{\Omega_{\delta}}(u_1(x,t_0)-u_2(x,t_0))_+\,dx
+\frac{C'\|h\|_{L^{\infty}}t_1\delta^{1-n}}{\delta^{2-n}-\delta_2^{2-n}}\sum_{i=1}^{i_0}\delta^{n-1-m\gamma_i'}.
\end{align}
Letting $t_0\to 0$ in \eqref{u1-2-l1-ineqn4}, by \eqref{initial-value-l1-sense} and \eqref{gamma-gamma'-lower-bd2}, 
\begin{align}\label{u1-2-l1-ineqn5}
&\int_{\Omega_{\delta}}(u_1(x,t_1)-u_2(x,t_1))h(x)\,dx
\le \frac{C'\|h\|_{L^{\infty}}t_1\delta^{2-n}}{\delta^{2-n}-\delta_2^{2-n}}\sum_{i=1}^{i_0}\delta^{n-2-m\gamma_i'}\notag\\
\Rightarrow\quad&\int_{\widehat{\Omega}}(u_1(x,t_1)-u_2(x,t_1))h(x)\,dx=0\quad\quad\forall t_1>0\quad\mbox{ as }\delta\to 0.
\end{align}
We now choose a sequence of smooth functions $0\le h_i\in C_0^{\infty}(\Omega_{\delta_2})$ such that $h_i(x)\to\chi_{\{u_1>u_2\}\cap\Omega_{\delta_2}}(x)$ for any $x\in \Omega_{\delta_2}$ as $i\to\infty$. Putting $h=h_i$ in \eqref{u1-2-l1-ineqn5} and letting $i\to\infty$,
\begin{align*}
&\int_{\Omega_{\delta_2}}(u_1(x,t_1)-u_2(x,t_1))_+\,dx=0\quad\forall t_1>0, 0<\delta_2<\delta_1/2\\
\Rightarrow\quad&\int_{\widehat{\Omega}}(u_1(x,t_1)-u_2(x,t_1))_+\,dx=0\quad\forall t_1>0\quad\mbox{ as }\delta_2\to 0
\end{align*}
and \eqref{u1-u2-compare} follows.
\end{proof}

By Theorem 1.1, Lemma 2.3, Lemma 2.15 of \cite{HK2} and Theorem \ref{uniqueness-thm2-bd-domain} and Proposition \ref{u-go-to-u0-l1-5} we have the following result. 

\begin{thm}\label{unique-soln-thm2-bd-domain}
Let $n\geq 3$, $0<m<\frac{n-2}{n}$, $0<\delta_1<\min (1,\delta_0)$, $\mu_0>0$, $f\in L^{\infty}(\partial\Omega\times (0,\infty))$ be such that $f\ge\mu_0$ on $\partial\Omega\times (0,\infty)$ and $\mu_0\leq u_0\in L_{loc}^p(\2{\Omega}\setminus\{a_1,\cdots,a_{i_0}\})$  for some constant  $p>\frac{n(1-m)}{2}$ be such that \eqref{u0-lower-blow-up-rate} and \eqref{upper-blow-up-rate-initial-data} 
hold for some constants satisfying \eqref{gamma-gamma'-lower-bd2} and $\lambda_1$, $\cdots$, $\lambda_{i_0}$, $\lambda_1'$, $\cdots$, $\lambda_{i_0}'\in\R^+$. 
Then there exists a unique solution $u$ of \eqref{Dirichlet-blow-up-problem} 
which satisfies  \eqref{initial-value-l1-sense2} such that for any constants $T>0$ and $\delta_2\in (0,\delta_1)$ there exist constants $C_1=C_1(T)>0$, $C_2=C_2(T)>0$, such that \eqref{u-lower-upper-blow-up-rate5} holds.
\end{thm}

\begin{proof}[\textbf{Proof of Theorem \ref{uniqueness-thm-Rn}}:]
Since the proof is similar to the proof of Theorem \ref{uniqueness-thm2-bd-domain}, we will only sketch the argument here. Let
\begin{equation}\label{a-defn}
A=A(x,t)=\left\{\begin{aligned}
&\frac{u_1(x,t)^m-u_2(x,t)^m}{u_1(x,t)-u_2(x,t)}\quad\forall x\in\widehat{\R^n}, t>0 \mbox{ satisfying }u_1(x,t)\ne u_2(x,t)\\
&mu_1(x,t)^{m-1}\qquad\qquad\forall x\in\widehat{\R^n}, t>0 \mbox{ satisfying }u_1(x,t)= u_2(x,t)\\
&0\qquad\qquad\qquad\qquad\,\,\forall x=a_i, i=1,\cdots, i_0,t>0.
\end{aligned}\right.
\end{equation}
For any $k\in\Z^+$, let
\begin{equation}\label{alpha-k-defn}
\alpha_k(x,t)=\left\{\begin{aligned}
&\frac{|u_1(x,t)^m-u_2(x,t)^m|}{|u_1(x,t)-u_2(x,t)|+(1/k)}\quad\forall x\in\widehat{\R^n}, t>0\\
&0\qquad\qquad\qquad\qquad\qquad\quad\forall x=a_i, i=1,\cdots, i_0,t>0
\end{aligned}\right.
\end{equation}
and $A_k=A_k(x,t)=\alpha_k(x,t)+k^{-1}$. Let $0<\delta_2\le\delta_1/2$. For any  $\delta\in (0,\delta_2/2)$ and $j\ge 2/\delta_2$, let $\phi$, $\phi_j$ and $w_j$ be as in the proof of Theorem \ref{uniqueness-thm2-bd-domain}. Let $t_1>t_0>0$, $R_0'>R_1+1$, $R>2R_0'$ and $h\in C_0^{\infty}(\Omega_{\delta_2,R_0'})$.
For any $k\in\Z^+$ and $0<\delta\le\delta_2/2$, let
$\psi_{k,\delta,R}$ be the solution of 
\begin{equation}\label{psi-k-delta-R-eqn}
\left\{\begin{aligned}
&\psi_t+A_k\Delta\psi=0\quad\mbox{ in }\Omega_{\delta,R}\times (0,t_1)\\
&\psi(x,t)=0\qquad\quad\,\mbox{on }\1\Omega_{\delta,R}\times (0,t_1)\\
&\psi(x,t_0)=h(x)\quad\,\mbox{ in }\Omega_{\delta,R}
\end{aligned}\right.
\end{equation}
By the maximum principle, $0\le \psi_{k,\delta,R}\le \|h\|_{L^{\infty}}$ in 
$\Omega_{\delta,R}\times (0,t_1)$. Hence 
$\frac{\1\psi_{k,\delta,R}}{\1\nu}\le 0$ on $\1 B_R\times (0,t_1)$. 
Then by an argument similar to the proof of Theorem \ref{uniqueness-thm2-bd-domain},
\begin{align}\label{u1-2-l1-Rn-ineqn3}
&\int_{\Omega_{\delta,R}}(u_1(x,t_1)-u_2(x,t_1))h(x)\,dx\notag\\
\le&\|h\|_{L^{\infty}}\int_{\Omega_{\delta,R}}(u_1(x,t_0)-u_2(x,t_0))_+\,dx
+\int_{t_0}^{t_1}\int_{\Omega_{\delta,R}}|u_1-u_2||A-A_k||\Delta \psi_{k,\delta,R}|\,dx\,dt\notag\\
&\qquad +\int_{t_0}^{t_1}\int_{\1 B_R}|u_1(x,t)^m-u_2(x,t)^m|\left|\frac{\1\psi_{k,\delta,R}}{\1\nu}\right|\,d\sigma\, dt\notag\\
&\qquad +\frac{C\delta^{2-n}}{\delta^{2-n}-\delta_2^{2-n}}\|h\|_{L^{\infty}}\sum_{i=1}^{i_0}\delta^{n-2-m\gamma_i'}.
\end{align}
Let 
\begin{equation*}
Q(x)=\frac{|x|^{2-n}-R^{2-n}}{(R/2)^{2-n}-R^{2-n}}\|h\|_{L^{\infty}}.
\end{equation*}
Then $Q$ satisfies
\begin{equation}\label{q-time-eqn2}
\left\{\begin{aligned}
q_t+A_k\Delta q=&0\qquad\quad\mbox{ in }(B_R\setminus\2{B_{R/2}})\times (0,t_1)\\
q=&0\qquad\quad\mbox{ on }\1 B_R\times (0,t_1)\\
q=&\|h\|_{L^{\infty}}\quad\mbox{ on }\1 B_{R/2}\times (0,t_1)\\
q\ge&0 \qquad\quad\mbox{ on }B_R\setminus B_{R/2}
\end{aligned}\right.
\end{equation}
Since $\psi_{k,\delta,R}$ is a subsolution of \eqref{q-time-eqn2}, by the maximum principle,
\begin{align}\label{psi-k-delta-r-derivative-bd}
&0\le \psi_{k,\delta,R}(x,t)\le Q(x)\quad\forall R/2\le |x|\le R, 0<t\le t_1\notag\\
\Rightarrow\quad&\left|\frac{\1\psi_{k,\delta,R}}{\1\nu}\right|\le\left|\frac{\1 Q}{\1\nu}\right|
=\frac{(n-2)R^{1-n}}{(R/2)^{2-n}-R^{2-n}}\|h\|_{L^{\infty}}\le\frac{C}{R}\|h\|_{L^{\infty}}\quad\mbox{ on }\1 B_R\times (0,t_1).
\end{align}
By \eqref{u1-2-l1-Rn-ineqn3} and \eqref{psi-k-delta-r-derivative-bd},
\begin{align}\label{u1-2-l1-Rn-ineqn4}
&\int_{\Omega_{\delta,R}}(u_1(x,t_1)-u_2(x,t_1))h(x)\,dx\notag\\
\le&\|h\|_{L^{\infty}}\int_{\Omega_{\delta,R}}(u_1(x,t_0)-u_2(x,t_0))_+\,dx
+\int_{t_0}^{t_1}\int_{\Omega_{\delta,R}}|u_1-u_2||A-A_k||\Delta \psi_{k,\delta,R}|\,dx\,dt\notag\\
&\qquad +\frac{C\|h\|_{L^{\infty}}}{R}\int_{t_0}^{t_1}\int_{\1 B_R}|u_1(x,t)^m-u_2(x,t)^m|\,d\sigma\, dt\notag\\
&\qquad +\frac{C\delta^{2-n}}{\delta^{2-n}-\delta_2^{2-n}}\|h\|_{L^{\infty}}\sum_{i=1}^{i_0}\delta^{n-2-m\gamma_i'}.
\end{align}
Letting first $k\to\infty$ and then $t_0\to 0$, $\delta\to 0$ in \eqref{u1-2-l1-Rn-ineqn4}, by the proof of Lemma 2.3 of \cite{DK} and similar argument as the proof of Theorem \ref{uniqueness-thm2-bd-domain}, the first term, second term and the last term on the right hand side of \eqref{u1-2-l1-Rn-ineqn4} vanish. This together with the mean value theorem and \eqref{u1-2-Rn-lower-bd} implies that
\begin{align}\label{u1-2-l1-Rn-ineqn5}
&\int_{\widehat{B}_R}(u_1(x,t_1)-u_2(x,t_1))h(x)\,dx\notag\\
\le &\frac{C\|h\|_{L^{\infty}}}{R}\int_{t_0}^{t_1}\int_{\1 B_R}|u_1(x,t)^m-u_2(x,t)^m|\,d\sigma\, dt\notag\\
\le &\frac{mC\mu_0^{m-1}\|h\|_{L^{\infty}}}{R}\int_0^{t_1}\int_{\1 B_R}|u_1(x,t)-u_2(x,t)|\,d\sigma\, dt\notag\\
\le &\frac{C'\|h\|_{L^{\infty}}}{R}\left\{\int_0^{t_1}\int_{\1 B_R}|u_1(x,t)-\mu_0|\,d\sigma\, dt
+\int_0^{t_1}\int_{\1 B_R}|u_1(x,t)-\mu_0|\,d\sigma\, dt\right\}.
\end{align}
By \eqref{uj-mu0-l1-difference} there exists a sequence $\{R_j\}_{j=2}^{\infty}\subset (2R_0',\infty)$, $R_j\to\infty$ as $j\to\infty$, such that
\begin{equation}\label{uj-mu0-l1-difference-go-0-at-infty}
\int_0^{t_1}\int_{\1 B_{R_j}}(|u_1(x,t)-\mu_0|+|u_2(x,t)-\mu_0|)\,d\sigma\, dt\to 0\quad\mbox{ as }j\to\infty.
\end{equation} 
Putting $R=R_j$ in \eqref{u1-2-l1-Rn-ineqn5} and letting $j\to\infty$, by \eqref{uj-mu0-l1-difference-go-0-at-infty},
\begin{equation}\label{u1-2-l1-ineqn-rn}
\int_{\widehat{\R^n}}(u_1(x,t_1)-u_2(x,t_1))h(x)\,dx=0\quad\forall t_1>0.
\end{equation}
By \eqref{u1-2-l1-ineqn-rn} and an argument similar to the proof of Theorem \ref{uniqueness-thm2-bd-domain},
\begin{equation*}
\int_{\widehat{\R^n}}(u_1(x,t_1)-u_2(x,t_1))_+\,dx=0\quad\forall t_1>0
\end{equation*}
and \eqref{u1-u2-Rn-compare} follows.
\end{proof}

By Theorem 1.2, Lemma 2.3, Lemma 2.15 and the proof of Theorem 1.6 of \cite{HK2} and Theorem \ref{uniqueness-thm-Rn}  we have the following result. 

\begin{thm}\label{unique-Rn-soln-thm}
Let $n\geq 3$, $0<m<\frac{n-2}{n}$, $0<\delta_1<\min (1,\delta_0)$, $\mu_0>0$ and $\mu_0\leq u_0\in L_{loc}^p(\widehat{\R^n}\setminus\{a_1,\cdots,a_{i_0}\})$  for some constant  $p>\frac{n(1-m)}{2}$ be such that 
\eqref{2initial-value-lower-upper-bd} holds for some constants satisfying \eqref{gamma-upper-lower-bd1} and $\lambda_1$, $\cdots$, $\lambda_{i_0}$, $\lambda_1'$, $\cdots$, $\lambda_{i_0}'\in\R^+$. Suppose \eqref{u0j-mu0-l1-difference} also holds for some constant $R_1>R_0$.
Then there exists a unique solution $u$ of \eqref{cauchy-blow-up-problem} which satisfy
\begin{equation*}
u(x,t)\ge\mu_0\quad\forall x\in \widehat{\R^n}, t>0
\end{equation*}
and
\begin{equation*}
\int_{\widehat{\R^n}}|u(x,t)-\mu_0|\,dx\le\int_{\widehat{\R^n}}|u_0-\mu_0|\,dx\quad\forall t>0
\end{equation*}
such that for any constants $T>0$ and $\delta_2\in (0,\delta_1)$ there exist constants $C_1=C_1(T)>0$, $C_2=C_2(T)>0$, such that \eqref{u-lower-upper-blow-up-rate5} holds.
\end{thm}

\section{Existence of highly oscillating solution}
\setcounter{equation}{0}
\setcounter{thm}{0}

In this section we will prove the existence of initial data such that the corresponding solution of \eqref{Dirichlet-blow-up-problem-constant-bdary-value} oscillates between infinity and some positive constant as $t\to\infty$. We start with a stability result for the solutions of \eqref{Dirichlet-blow-up-problem-constant-bdary-value}.

\begin{lem}\label{stability-thm-bd-domain}
Let $n\geq 3$, $0<m<\frac{n-2}{n}$, $0<\delta_1<\min (1,\delta_0)$, $\mu_0>0$. Let  $\{u_{0,j}\}_{j=1}^{\infty}\subset L_{loc}^p(\2{\Omega}\setminus\{a_1,\cdots,a_{i_0}\})$ for some constant  $p>\frac{n(1-m)}{2}$ be a sequence of functions  satisfying 
\begin{equation}\label{u-0-j-lower-bd}
u_{0,j}\ge\mu_0\quad\mbox{ on }\2{\Omega}\setminus\{a_1,\cdots,a_{i_0}\}\quad\forall j\in\Z^+
\end{equation}
such that 
\begin{equation}\label{initial-value-lower-upper-bd2-1}
\frac{\lambda_i}{|x-a_i|^{\gamma_i}}\le u_{0,j}(x)\le \frac{\lambda_i'}{|x-a_i|^{\gamma_i'}} \qquad \forall 0<|x-a_i|<\delta_1,i=1, \cdots, i_0, j\in\Z^+
\end{equation}
holds for some constants  satisfying \eqref{gamma-gamma'-lower-bd}, $\lambda_1$, $\cdots$, $\lambda_{i_0}$, $\lambda_1'$, $\cdots$, $\lambda_{i_0}'\in\R^+$. Let $\mu_0\leq u_0\in L_{loc}^p(\2{\Omega}\setminus\{a_1,\cdots,a_{i_0}\})$  be such that \eqref{initial-value-lower-upper-bd} holds. Let $u$, $u_j$, $j\in\Z^+$, be the unique solutions of \eqref{Dirichlet-blow-up-problem-constant-bdary-value}
 with initial value $u_0$, $u_{0,j}$ respectively, given by Theorem \ref{unique-soln-thm-bd-domain}. 
Suppose 
\begin{equation}\label{u-0j-converge-to-u0}
u_{0,j}\to u_0\quad\mbox{ in }L_{loc}^p(\2{\Omega}\setminus\{a_1,\cdots,a_{i_0}\})\quad\mbox{ as }j\to\infty.
\end{equation}
Then $u_j$ converges to $u$ uniformly in $C^{2,1}(\Omega_{\delta}\times (t_1,t_2))$ as $j\to\infty$ for any $0<\delta<\delta_0$ and $t_2>t_1>0$.
\end{lem}
\begin{proof}
Let $0<\delta'<\delta<\delta_0$ and $t_2>t_1>0$. By \eqref{u-0j-converge-to-u0} there exists a constant $M_1>0$ such that
\begin{equation}\label{u0j-uniform-lp-loc}
\|u_{0,j}\|_{L^p(\Omega_{\delta'})}\le M_1\quad\forall j\in\Z^+.
\end{equation}
By \eqref{u0j-uniform-lp-loc} and Lemma 2.9 of \cite{HK2} there exists a constant $M_2>0$ depending on $M_1$ and $\mu_0$ such that
\begin{equation}\label{uj-loc-uniform-upper-bd}
\|u_j\|_{L^{\infty}(\Omega_{\delta}\times (t_1/2,t_2))}\le M_2\quad\forall j\in\Z^+.
\end{equation}
By Theorem \ref{unique-soln-thm-bd-domain},
\begin{equation}\label{uj-uniform-lower-bd}
u_j\ge\mu_0\quad\mbox{ in }\2{\Omega}\setminus\{a_1,\cdots,a_{i_0}\}\times (0,\infty)\quad\forall j\in\Z^+.
\end{equation}
By \eqref{uj-loc-uniform-upper-bd} and \eqref{uj-uniform-lower-bd} the equation \eqref{fde} for $u_j$ are uniformly parabolic on every compact subset of $\2{\Omega}\setminus\{a_1,\cdots,a_{i_0}\}\times (0,\infty)$. Hence by the Ascoli theorem,  diagonalization argument, and an argument similar to the proof of Lemma 2.11 of \cite{HK2}  
and Theorem 1.1 of \cite{HK1}  the sequence $\{u_j\}_{j=1}^{\infty}$ has a subsequence $\{u_{j_k}\}_{k=1}^{\infty}$ 
that converges uniformly in $C^{2,1}(\Omega_{\delta}\times (t_1,t_2))$ to a solution $v$ of 
\eqref{Dirichlet-blow-up-problem-constant-bdary-value} as $k\to\infty$ for any $0<\delta<\delta_0$ and $t_2>t_1>0$ and
\begin{equation}\label{v-lower-bd}
v\ge\mu_0\quad\mbox{ in }\widehat{\Omega}\times (0,\infty). 
\end{equation}
Since by Theorem \ref{unique-soln-thm-bd-domain} for any $T>0$ there exists constants $C_1=C_1(T)>0$, $C_2=C_2(T)>0$, such that \eqref{u-lower-upper-blow-up-rate5} holds for any $u_j$, putting $u=u_{j_k}$ in \eqref{u-lower-upper-blow-up-rate5} and letting $k\to\infty$,
\begin{equation}\label{v-lower-upper-blow-up-rate5}
\frac{C_1}{|x-a_i|^{\gamma_i}}\le v(x,t)\le\frac{C_2}{|x-a_i|^{\gamma_i'}} \quad \forall 0<|x-a_i|<\delta_2, 0<t<T, i=1,2,\dots,i_0.
\end{equation}
By Lemma \ref{u-local-l1-bd-intial-value-lem} there exists a constant $C>0$ such that 
\begin{align}\label{uj-integral-bd-by-time-0-integral}
\int_{D_{\delta_1}}u_j(x,t)\,dx
\le&\left\{\left(\int_{D_{2\delta_1}}u_{0,j}\,dx\right)^{1-m}+Ct\right\}^{\frac{1}{1-m}}+|D_{\delta_1}|\mu_0\quad\forall t>0, 0<\delta_1<\delta_0/2,j\in\Z^+\notag\\
\le&\left\{\left(|D_{2\delta_1}|^{1-\frac{1}{p}}\|u_{0,j}\|_{L^p(D_{2\delta_1})}\right)^{1-m}+Ct\right\}^{\frac{1}{1-m}}+|D_{\delta_1}|\mu_0\quad\forall t>0, 0<\delta_1<\delta_0/2,j\in\Z^+.
\end{align}
Let $\3>0$. By \eqref{u-0j-converge-to-u0} there exists $j_0\in\Z^+$ such that
\begin{equation}\label{u-0j-upper-bd2}
\|u_{0,j}\|_{L^p(D_{\delta_1})}\le\|u_0\|_{L^p(D_{\delta_1})}+\3\quad\forall 0<\delta_1<\delta_0, j\ge j_0.
\end{equation}
By \eqref{uj-integral-bd-by-time-0-integral}, \eqref{u-0j-upper-bd2} and Holder's inequality,
\begin{align}\label{uj-u0j-difference10}
&\int_{\Omega_{\delta}}|u_j(x,t)-u_{0,j}(x)|\,dx\notag\\
\le&\int_{\Omega_{\delta}\setminus D_{\delta_1}}|u_j(x,t)-u_{0,j}(x)|\,dx+\left\{\left(|D_{2\delta_1}|^{1-\frac{1}{p}}(\|u_0\|_{L^p(D_{2\delta_1})}+\3)\right)^{1-m}+Ct\right\}^{\frac{1}{1-m}}+|D_{\delta_1}|\mu_0\notag\\
&\qquad +|D_{\delta_1}|^{1-\frac{1}{p}}(\|u_0\|_{L^p(D_{\delta_1})}+\3)\quad\forall 0<\delta_1<\delta_0/2, t>0, j\ge j_0.
\end{align}
Leting $j=j_k\to\infty$ in \eqref{uj-u0j-difference10},
\begin{align}\label{v-u0-difference}
&\int_{\Omega_{\delta}}|v(x,t)-u_0(x)|\,dx\notag\\
\le&\int_{\Omega_{\delta}\setminus D_{\delta_1}}|v(x,t)-u_0(x)|\,dx+\left\{\left(|D_{2\delta_1}|^{1-\frac{1}{p}}(\|u_0\|_{L^p(D_{2\delta_1})}+\3)\right)^{1-m}+Ct\right\}^{\frac{1}{1-m}}+|D_{\delta_1}|\mu_0\notag\\
&\qquad +|D_{\delta_1}|^{1-\frac{1}{p}}(\|u_0\|_{L^p(D_{\delta_1})}+\3)\quad\forall 0<\delta_1<\delta_0/2, t>0.
\end{align}
Letting first $t\to 0$ and then  $\delta_1\to 0$ in \eqref{v-u0-difference},
\begin{equation}\label{v-go-to-u0}
\lim_{t\to 0}\int_{\Omega_{\delta}}|v(x,t)-u_0(x)|\,dx=0\quad\forall 0<\delta<\delta_0.
\end{equation}
By \eqref{v-lower-bd}, \eqref{v-lower-upper-blow-up-rate5}, \eqref{v-go-to-u0} and Theorem \ref{unique-soln-thm-bd-domain}, $v=u$ in $(\2{\Omega}\setminus\{a_1,\cdots,a_{i_0})\times (0,\infty)$. Hence $u_j$ converges to $u$ uniformly in $C^{2,1}(\Omega_{\delta}\times (t_1,t_2))$ as $j\to\infty$ for any $0<\delta<\delta_0$ and $t_2>t_1>0$ and the lemma follows.
\end{proof}

We next recall two results from \cite{HK2}.

\begin{thm}\label{convergence-thm1}(cf. Theorem 1.3 of \cite{HK2})
Suppose that $n\geq 3$, $0<m<\frac{n-2}{n}$ and $\mu_0>0$. Let $\mu_0\leq u_0\in L^{p}_{loc}\left(\overline{\Omega}\bs\left\{a_1,\cdots,a_{i_0}\right\}\right)$ for some constant $p>\frac{n(1-m)}{2}$ satisfy \eqref{initial-value-lower-upper-bd} for some constants satisfying \eqref{gamma-upper-lower-bd1} and $\lambda_1$, $\cdots$, $\lambda_{i_0}$, $\lambda_1'$, $\cdots$, $\lambda_{i_0}'\in\R^+$.
Let $u$ be the solution of \eqref{Dirichlet-blow-up-problem-constant-bdary-value} given by Theorem \ref{unique-soln-thm-bd-domain}. Then
\begin{equation}\label{u-infty-to-mu0}
u(x,t)\to \mu_0 \quad \mbox{ in }C^2(K) \quad \mbox{ as }t\to\infty
\end{equation}
for any compact subset $K$ of $\overline{\Omega}\bs\left\{a_1,\cdots,a_{i_0}\right\}$.
\end{thm}

\begin{thm}\label{convergence-thm6}
Suppose that $n\geq 3$, $0<m<\frac{n-2}{n}$ and $\mu_0>0$. Let $\mu_0\leq u_0\in L^{p}_{loc}\left(\overline{\Omega}\bs\left\{a_1,\cdots,a_{i_0}\right\}\right)$ for some constant $p>\frac{n(1-m)}{2}$ satisfy \eqref{u0-lower-blow-up-rate} for some constants satisfying 
\begin{equation}\label{gamma1-large}
\gamma_1>\frac{n-2}{m},\quad
\gamma_i>\frac{2}{1-m}\quad\forall i=2,\dots,i_0,
\end{equation}
 and $0<\delta_1<\delta_0$, $\lambda_1$, $\cdots$, $\lambda_{i_0}\in\R^+$. Let $u$ be the solution of \eqref{Dirichlet-blow-up-problem-constant-bdary-value} given by Theorem \ref{unique-soln-thm-bd-domain}. Then
\begin{equation}\label{eq-blow-up-of-soluiton-u-when-gamma-greater-than-n-2-over-m-2c}
u(x,t)\to\infty \quad \mbox{ on } K\quad \mbox{ as }t\to\infty
\end{equation}
for any compact subset $K$ of $\widehat{\Omega}$. 
\end{thm}

\begin{proof}[\textbf{Proof of Theorem \ref{existence-oscillating-soln-thm}}:]
We will use a modification of the proof of Theorem 1 of \cite{VW2} to construct the oscillating solution $u$ of \eqref{Dirichlet-blow-up-problem-constant-bdary-value} as the limit of a sequence of solutions $u_j$ of \eqref{Dirichlet-blow-up-problem-constant-bdary-value} with initial value $u_{0,j}$ that satisfy appropriate blow-up condition at the points $a_1,\cdots,a_{i_0}$. 
Let 
\begin{equation*}
\alpha_1>\frac{n-2}{m},\quad\alpha_2=\frac{\frac{2}{1-m}+n}{2},
\end{equation*}
and let $K$ be a compact subset of $\widehat{\Omega}$. We choose $j_1\in\Z^+$ such that $j_1>\max(\delta_0^{-1},\mu_0^{1/\alpha_1},\mu_0^{1/\alpha_2})$. Let
\begin{equation*}
u_{0,1}(x)=\left\{\begin{aligned}
&j_1^{\alpha_2}\qquad\qquad\forall x\in\2{\Omega}\setminus\cup_{i=1}^{i_0}B_{1/j_1}(a_i)\\
&|x-a_i|^{-\alpha_2}\quad\,\,\forall x\in B_{1/j_1}(a_i), i=1,\cdots,i_0.
\end{aligned}\right.
\end{equation*}
Then $u_{0,1}(x)\ge\mu_0$ for any $x\in \overline{\Omega}\bs\left\{a_1,\cdots,a_{i_0}\right\}$.
By Theorem \ref{unique-soln-thm-bd-domain} there exists a unique solution $u_1$ of \eqref{Dirichlet-blow-up-problem-constant-bdary-value} which satisfies $u_1\ge\mu_0$ in $(\overline{\Omega}\bs\left\{a_1,\cdots,a_{i_0}\right\})\times (0,\infty)$. By Theorem \ref{convergence-thm1},
\begin{equation*}
u_1(x,t)\to \mu_0 \quad \mbox{ in }C^2(K) \quad \mbox{ as }t\to\infty.
\end{equation*}
Hence there exists a constant $t_1>1$ such that
\begin{equation}\label{u1-time-t1-near-mu0}
\mu_0\le u_1(x,t_1)\le\mu_0+\frac{1}{2}\quad\forall x\in K.
\end{equation}
For any $j\in\Z^+$, $j>j_1$, let
\begin{equation}\label{u-01j-behaviour-near-ai}
u_{0,1,j}(x)=\left\{\begin{aligned}
&u_{0,1}(x)\qquad\quad\forall x\in\2{\Omega}\setminus B_{1/j}(a_1)\\
&|x-a_1|^{-\alpha_1}\quad\,\,\,\forall x\in B_{1/j}(a_1).
\end{aligned}\right.
\end{equation}
Then $u_{0,1,j}(x)\ge\mu_0$ for any $x\in \overline{\Omega}\bs\left\{a_1,\cdots,a_{i_0}\right\}$ and $j>j_1$.
For any $j>j_1$, let $u_{2,j}$ be the unique solution of \eqref{Dirichlet-blow-up-problem-constant-bdary-value} with $u_0=u_{0,1,j}$ given by Theorem \ref{unique-soln-thm-bd-domain} which satisfies $u_{2,j}\ge\mu_0$ in $(\overline{\Omega}\bs\left\{a_1,\cdots,a_{i_0}\right\})\times (0,\infty)$. Since $u_{0,1,j}$ converges to $u_{0,1}$ in $L^{p}_{loc}\left(\overline{\Omega}\bs\left\{a_1,\cdots,a_{i_0}\right\}\right)$ as $j\to\infty$, by Lemma \ref{stability-thm-bd-domain}
 $u_{2,j}(x,t_1)$ converges to $u_1(x,t_1)$ uniformly in $K$ as $j\to\infty$. Hence there exists $j_2\in\Z^+$,$j_2>j_1$, such that
\begin{equation}\label{u2-u1-time-t1-difference}
|u_{2,j_2}(x,t_1)-u_1(x,t_1)|\le\frac{1}{4}\quad\forall x\in K.
\end{equation} 
Let $u_2=u_{2,j_2}$ and $u_{0,2}=u_{0,1,j_2}$. 
By \eqref{u1-time-t1-near-mu0} and \eqref{u2-u1-time-t1-difference},
\begin{equation*}\label{u2-time-t1-near-mu0}
\mu_0\le u_2(x,t_1)\le\mu_0+\frac{3}{4}\quad\forall x\in K.
\end{equation*}
By \eqref{u-01j-behaviour-near-ai} and Theorem \ref{convergence-thm6}, 
\begin{equation*}
u_2(x,t)\to\infty\quad \mbox{ in }C^2(K) \quad \mbox{ as }t\to\infty.
\end{equation*}
Hence there exists a constant $t_2>t_1+1$ satisfying
\begin{equation*}\label{u2-time-t2-infty}
u_2(x,t_2)\ge 3\quad\forall x\in K.
\end{equation*}
Repeating the above argument we get sequences $\{u_{0,k}\}_{k=1}^{\infty}\subset L^{p}_{loc}\left(\overline{\Omega}\bs\left\{a_1,\cdots,a_{i_0}\right\}\right)$, $\{j_k\}_{k=1}^{\infty}\subset\Z^+$ and $\{t_k\}_{k=1}^{\infty}\subset \R^+$, such that $j_{k+1}>j_k$ and $t_{k+1}>t_k+1$ for all $k\in\Z^+$, which satisfy
\begin{equation*}\label{u0k-uniform-lower-bd}
u_{0,k}(x)\ge\mu_0\quad\mbox{ in }\overline{\Omega}\bs\left\{a_1,\cdots,a_{i_0}\right\}\quad\forall k\in\Z^+, 
\end{equation*}
\begin{equation*}\label{uk=u-k-1}
u_{0,k}(x)=u_{0,k-1}(x)\quad\forall x\in\2{\Omega}\setminus B_{1/j_k}(a_1),k\ge 2
\end{equation*}
and
\begin{equation*}\label{u0j-behaviour-near-ai}
u_{0,k}(x)=\left\{\begin{aligned}
&|x-a_1|^{-\alpha_2}\quad\,\,\forall x\in B_{1/j_k}(a_1)\quad\mbox{ if $k\ge 1$ is odd}\\
&|x-a_1|^{-\alpha_1}\quad\,\,\forall x\in B_{1/j_k}(a_1)\quad\mbox{ if $k\ge 2$ is even}
\end{aligned}\right.
\end{equation*}
and if $u_k$ is the solution of \eqref{Dirichlet-blow-up-problem-constant-bdary-value} with $u_0=u_{0,k}$ given by Theorem \ref{unique-soln-thm-bd-domain}, then
$u_k$ satisfies
\begin{equation}\label{uk-odd-time-value}
\mu_0\le u_k(x,t_l)\le\mu_0+\frac{1}{2^l}+\cdots+\frac{1}{2^k}\le\mu_0+\frac{1}{2^{l-1}}\quad\forall x\in K, 1\le l\le k\mbox{ and $l$ is odd},
\end{equation}
\begin{equation}\label{uk-even-time-value}
u_k(x,t_l)>l+\frac{3}{2}-\left(\frac{1}{2^l}+\cdots+\frac{1}{2^k}\right)>l\quad\forall x\in K, 1\le l\le k\mbox{ and $l$ is even}
\end{equation}
and
\begin{equation*}
|u_k(x,t_l)-u_{k+1}(x,t_l)|<\frac{1}{2^k}\quad\forall x\in K, 1\le l\le k,k\in\Z^+.
\end{equation*}
Let
\begin{equation*}
u_0(x)=\left\{\begin{aligned}
&j_1^{\alpha_2}\qquad\qquad\forall x\in\2{\Omega}\setminus\cup_{i=1}^{i_0}B_{1/j_1}(a_i)\\
&|x-a_i|^{-\alpha_2}\quad\,\forall x\in B_{1/j_1}(a_i), i=2,\cdots,i_0\\
&u_{0,k}(x)\qquad\,\,\,\forall 1/j_{k+1}\le |x-a_1|\le 1/j_k, k\ge\Z^+.
\end{aligned}\right.
\end{equation*}
Then $u_0\ge\mu_0$ in $\overline{\Omega}\bs\left\{a_1,\cdots,a_{i_0}\right\}$, 
\begin{equation*}
u_0(x)=\left\{\begin{aligned}
&|x-a_1|^{-\alpha_2}\quad\forall 1/j_{k+1}\le |x-a_1|\le 1/j_k, k\ge\Z^+\quad\mbox{ and $k$ is odd}\\
&|x-a_1|^{-\alpha_1}\quad\forall 1/j_{k+1}\le |x-a_1|\le 1/j_k, k\ge\Z^+\quad\mbox{ and $k$ is even},
\end{aligned}\right.
\end{equation*}
and $u_{0,k}$ converges to $u_0$ in $L^{p}_{loc}\left(\overline{\Omega}\bs\left\{a_1,\cdots,a_{i_0}\right\}\right)$ as $k\to\infty$. let $u$ be the unique solution of \eqref{Dirichlet-blow-up-problem-constant-bdary-value} given by Theorem \ref{unique-soln-thm-bd-domain}. Then by Lemma \ref{stability-thm-bd-domain} $u_k$ converges to $u$ on every compact subset of $\widehat{\Omega}\times (0,\infty)$ as $k\to\infty$. letting $k\to\infty$ in 
\eqref{uk-odd-time-value} and \eqref{uk-even-time-value} we have
\begin{align*}
&\mu_0\le u(x,t_l)\le\mu_0+\frac{1}{2^{l-1}}\quad\forall x\in K, l\in\Z^+\mbox{ and $l$ is odd}\\
\Rightarrow\quad&\lim_{k\to\infty}u(x,t_{2k-1})=\mu_0\quad\mbox{ uniformly in }K
\end{align*}
and
\begin{align*}
&u(x,t_l)\ge l\quad\forall x\in K, l\in\Z^+\mbox{ and $l$ is even}\\
\Rightarrow\quad&\lim_{k\to\infty}u(x,t_{2k})=\infty\quad\mbox{ uniformly in }K
\end{align*}
and the theorem follows.
\end{proof}


\begin{thebibliography}{99}

\bibitem[A]{A} D.G.~Aronson, {\em The porous medium equation}, CIME
Lectures in Some problems in Nonlinear Diffusion, Lecture
Notes in Mathematics 1224, Springer-Verlag, New York, 1986.

\bibitem[CL]{CL} B.~Choi and K.~Lee, {\it Multi-D fast diffusion equation via diffusive scaling of generalized Carleman kinetic equation}, arxiv:1510.08997.

\bibitem[DaK]{DaK} B.E.J.~Dahlberg and C.~Kenig, {\em  Non-negative
solutions of generalized porous medium equations}, Revista
Matem\'atica Iberoamericana 2 (1986), 267--305.
 
\bibitem[DK]{DK} P.~Daskalopoulos and C.E.~Kenig, {\em Degenerate
diffusion-initial value problems and local regularity theory},
Tracts in Mathematics 1, European Mathematical Society, 2007.

\bibitem[DS1]{DS1} P.~Daskalopoulos and N.~Sesum, {\em On the extinction profile of solutions to fast diffusion}, J. Reine Angew. Math. 622 (2008), 95--119.

\bibitem[DS2]{DS2} P.~Daskalopoulos and N.~Sesum, {\em The classification of locally conformally flat Yamabe solitons}, Advances in Math. 240 (2013), 346-369.

\bibitem[GS]{GS} F.~Golse and F.~Salvarani, {\it The nonlinear diffusion limit for generalized Carleman models: the initial-boundary value problem}, Nonlinearity 20 (2007), 927-942.

\bibitem[HK1]{HK1} K.M.~Hui and S.~Kim, {\it Existence of Neumann and singular solutions of the fast diffusion equation}, Discrete Contin. Dyn. Syst. Series-A 35 (2015), no. 10, 4859-4887. 

\bibitem[HK2]{HK2} K.M.~Hui and S.~Kim, {\it Existence and large time behaviour of finite points
blow-up solutions of the fast diffusion equation}, arxiv:1712.05515v2.

\bibitem[HKs]{HKs} K.M. Hui and Soojung Kim, {\it Asymptotic large time behaviour of singular solutions of the fast diffusion equation}, Discrete Contin. Dyn. Syst. Series-A 37 (2017), no. 11, 5943--5977.

\bibitem[HP]{HP} M.A.~Herrero and M.~Pierre, {\it The Cauchy problem for $u_t=\Delta u^m$ when $0<m<1$}, Transactions A. M. S. 291 (1985), no. 1, 145--158.  

\bibitem[HPW]{HPW} F.~Huang, R.~Pan and Z.~Wang, {\it $L^1$ convergence to the Barenblatt solution for compressibe Euler equations with damping}, Arch. Ration. Mech. Anal. 200 (2011), no. 2, 665-689.

\bibitem[Hs]{Hs} S.Y. Hsu, {\it Existence and asymptotic behaviour of solutions of the very fast diffusion equation}, Manuscripta Math. 140 (2013), no. 3-4, 441-460.

\bibitem[K]{K} T. Kato,  {\it Schr\"odinger operators with singular potentials},  Israel J. Math. 13 (1973),  135--148.

\bibitem[LSU]{LSU} O.A.~Ladyzenskaya, V.A.~Solonnikov and N.N.~Uraltceva, {\it Linear and quasilinear equations of parabolic type}, Transl. Math. Mono. vol. 23, Amer. Math. Soc., Providence, R.I., U.S.A., 1968.

\bibitem[LZ]{LZ} T.~Luo and H.~Zeng, {\em Global existence and smooth solutions and convergence to Barenblatt solutions for the physical vacuum free boundary problem of compressible Euler equations with damping}, Commun. on Pure and Applied Math.  69 (2016), no. 7, 1354--1396.

\bibitem[PS]{PS} M. del Pino and M.~S\'aez, {\em On the extinction profile for solutions of $u_t=\Delta u^{(N-2)/(N+2)}$}, Indiana Univ. Math. J. 50
(2001), no. 1, 611-628.

\bibitem[PV]{PV} A.de~Pablo and J.L.~Vazquez, {\em Travelling waves and finite propagation in a reaction-diffusion equation}, J. Differential Equations 93 (1991), 19--61.

\bibitem[V1]{V1} J.L. Vazquez, {\it Smoothing and Decay Estimates for Nonlinear Diffusion Equations}, Oxford Lecture Series in Mathematics and its Applications 33, Oxford University Press, Oxford, 2006.

\bibitem[V2]{V2}  J.L. Vazquez, {\it The porous medium equation. Mathematical theory.} Oxford Mathematical Monographs. The Clarendon Press, Oxford University Press, Oxford, 2007. xxii+624 pp.

\bibitem[VW1]{VW1} J.L. Vazquez and M. Winkler, {\it The evolution of singularities in fast diffusion equations: infinite-time blow-down}, SIAM J. Math. Anal. 43 (2011), no. 4, 1499-1535. 

\bibitem[VW2]{VW2} J.L. Vazquez and M. Winkler, {\it Highly time-oscillating solutions for very fast diffusion equations.} J. Evol. Equ. 11 (2011), no. 3, 725-742.

\end{thebibliography}
\end{document}